\documentclass[11pt]{amsart}  
\usepackage{amssymb}  
\usepackage{amsthm}
\usepackage{thmtools}

\usepackage{amscd}
\usepackage{mathrsfs}
\usepackage{amsfonts}
\usepackage{amsmath}
\usepackage{graphicx}
\usepackage{enumerate}
\usepackage{pinlabel}
\usepackage{cite}
\usepackage{hyperref}
\usepackage[usenames]{xcolor}

\usepackage[normalem]{ulem}

\newcommand{\bz}{\mathbb Z}

\newcommand{\br}{\mathbb R}

\newcommand{\filt}{\mathcal{F}}

\DeclareMathOperator{\CKh}{CKh}
\DeclareMathOperator{\SKh}{AKh}
\DeclareMathOperator{\Kh}{Kh}

\newcommand{\diagram}{\mathcal{D}}

\newcommand{\ul}{\underline}

\DeclareMathOperator{\coker}{coker}

\DeclareMathOperator{\reducedsub}{\widetilde{CKh
}}
\DeclareMathOperator{\reducedquo}{\underline{CKh}}
\DeclareMathOperator{\subkapp}{\tilde{\kappa}_p}
\DeclareMathOperator{\quotkapp}{\underline{\kappa}_p}
\DeclareMathOperator{\subpsip}{\widetilde{\psi}_p}
\DeclareMathOperator{\quotpsip}{\underline{\psi}_p}

\newtheorem{Thm}{Theorem}
\newtheorem*{Thm*}{Theorem}
\newtheorem{Prop}[Thm]{Proposition}
\newtheorem{Lem}[Thm]{Lemma}
\newtheorem{Cor}[Thm]{Corollary}

\newtheorem*{Prop*}{Proposition}

\theoremstyle{definition}
\newtheorem{Def}[Thm]{Definition}
\newtheorem*{Def*}{Definition}

\theoremstyle{remark}

\newtheorem*{Rem*}{Remark}

\newtheorem*{Ex*}{Example}

\theoremstyle{definition}

\begin{document}  

\author{Diana Hubbard and Adam Saltz}
\title[An annular refinement of the transverse element in $\Kh$]{An annular refinement of the transverse element in Khovanov homology}
\maketitle
\thispagestyle{empty}

\begin{abstract} 
We construct a braid conjugacy class invariant $\kappa$ by refining Plamenevskaya's transverse element $\psi$ in Khovanov homology via the annular grading.  While $\kappa$ is not an invariant of transverse links, it distinguishes some braids whose closures share the same classical invariants but are not transversely isotopic.  Using $\kappa$ we construct an obstruction to negative destabilization (stronger than $\psi$) and a solution to the word problem in braid groups.  Also, $\kappa$ is a lower bound on the length of the spectral sequence from annular Khovanov homology to Khovanov homology, and we obtain concrete examples in which this spectral sequence does not collapse immediately.  In addition, we study these constructions in reduced Khovanov homology and illustrate that the two reduced versions are fundamentally different with respect to the annular filtration.
\end{abstract}

\section{Introduction}

Khovanov homology has proven to be a powerful tool for studying links and link cobordisms in $S^3$.  Given a link $L$ with diagram $\diagram$, the homology of the bigraded Khovanov chain complex $\CKh(\diagram)$ is a link invariant denoted $\Kh(L)$.  In \cite{plamenevskaya2004transverse}, Plamenevskaya constructs from Khovanov homology an invariant of transverse links presented as braid closures.\footnote{Recall that Orevkov and Shevchishin \cite{OrevShev}, and independently Wrinkle \cite{Wrinkle}, have shown that there is a one-to-one correspondence between transverse links in $S^3$ (up to isotopy) and braids (up to positive stabilization and isotopy).  We review this correspondence in Section 2.}  Recall that the Khovanov chain complex is constructed by assigning a vector space to each complete resolution of a diagram.  This vector space over $\bz/2\bz$ is generated by labelings of the components in each resolution by the symbols $v_+$ and $v_-$.  The diagram of an $n$-strand braid closure $L$ has a unique resolution into an $n$-strand unlink, and the transverse element $\psi(L)$ is the labeling of each component of this unlink by $v_-$.  It is easy to see that $\psi(L)$ is a cycle.

\begin{Thm*}
\cite{plamenevskaya2004transverse} Let $L$ and $L'$ be transversely isotopic transverse links with braid closure diagrams $\diagram$ and $\diagram'$, respectively.  Then any sequence of transverse Reidemeister moves connecting the two diagrams induces a map $\CKh(\diagram) \to \CKh(\diagram')$ which sends $\psi(L)$ to $\psi(L')$.
\end{Thm*}

This implies that the homology class $	[\psi(L)]$ (and indeed, the chain $\psi(L)$) is a transverse invariant which detects the classical self-linking number.  An invariant of transverse links is called \emph{effective} if it takes on different values for a pair of transverse links with the same self-linking number and smooth link type.  A smooth link type is called \emph{transversely non-simple} if it supports transversely non-isotopic links with the same self-linking number.  It is not known if Plamenevskaya's invariant is effective.

Given a link $L$ equipped with an embedding into a thickened annulus (i.e. $L \subset A \times I \subset S^{3}$), its Khovanov chain complex can be endowed with an additional grading which we call the $k$-grading, first studied by \cite{apsSKH} and \cite{robertsSKH}.  For a resolution with a single component, $k(v_{\pm}) = \pm 1$ if the component is not null-homotopic in $A \times I$, and $k(v_{\pm}) = 0$ otherwise.  We extend the grading to tensor products by summation.  The Khovanov differential is non-increasing in the $k$-grading, which induces a filtration on the Khovanov complex. The homology of the associated graded chain complex is called \emph{annular Khovanov homology}, denoted here as $\SKh(L)$ (elsewhere also called \emph{sutured annular Khovanov homology} or \emph{sutured Khovanov homology} and denoted SKh($L$)).  $\SKh$ is an invariant of annular links and not a transverse invariant.  (See Section \ref{preliminaries} for more details.)  For a braid closure $\bar{\beta}$, the element $\psi(\bar{\beta}) \in \CKh(\bar{\beta})$ is the unique element with lowest $k$-grading.  

Standard algebraic machinery (see \cite{hutchings2011introduction} for an introduction and \cite{Mccleary} for a thorough treatment) produces a spectral sequence from the associated graded object of a filtered complex to the homology of that complex and therefore from $\SKh$ to $\Kh$.  Our original goal in this work was to define a (perhaps effective) transverse invariant by exploring the behavior of Plamenevskaya's class in this spectral sequence. $\SKh$ is known to distinguish some braids whose closures are smoothly isotopic but not transversely isotopic (see \cite{hubbard2015sutured}), and so it is natural to suspect that the spectral sequence from $\SKh$ to $\Kh$ also captures non-classical information.

In this paper we define a refinement of Plamenevskaya's invariant that measures how long $\psi(L)$ survives in the spectral sequence, or equivalently, the lowest filtration level at which the class of $\psi(L)$ vanishes.  For a braid $\beta$ with closure $\bar\beta$, write $\filt_i(\bar\beta) = \{x \in \CKh(\bar\beta) : k(x) \leq i\}$.  

\begin{Def}\label{def:kappa} Let $\beta$ be an $n$-strand braid with closure $\bar\beta$ and suppose that $\psi(\bar\beta)$ is a boundary in $\CKh(\bar\beta)$.  Define \[ \kappa(\beta) = n + \min \{i :  [\psi(\bar\beta)] = 0 \in H(\filt_i)\}. \]  If $\psi(\bar\beta)$ is not a boundary then define $\kappa(\beta) = \infty$.\end{Def}

However, $\kappa(\beta)$ is a conjugacy class invariant of braids rather than a transverse invariant.

\begin{Thm} 
$\kappa$ is an invariant of conjugacy classes in the braid group $B_n$.  It may increase by $2$ under positive stabilization and is thus not a transverse invariant.  
\end{Thm}

Nevertheless, $\kappa$ can distinguish conjugacy classes of some braids whose closures are transversely non-isotopic but have the same classical invariants.  

\begin{Prop}
\label{ngexample} For any $a,b \in \{0,1,2\}$, the pair of closed $4$-braids
\[ A(a,b) = \sigma_{3}\sigma_{2}^{-2}\sigma_{3}^{2a+2}\sigma_{2}\sigma_{3}^{-1}\sigma_{1}^{-1}\sigma_{2}\sigma_{1}^{2b+2} \,\,\,\,\,\text{and}\,\,\, \]
\[ B(a,b) = \sigma_{3}\sigma_{2}^{-2}\sigma_{3}^{2a+2}\sigma_{2}\sigma_{3}^{-1}\sigma_{1}^{2b+2}\sigma_{2}\sigma_{1}^{-1}, \]
related by a negative flype, can be distinguished by $\kappa$: indeed, $\kappa(A(a,b)) = 4$ and $\kappa(B(a,b)) = 2$.  For any pair $(a,b)$, the braids $A(a,b)$ and $B(a,b)$ are transversely non-isotopic but have the same classical invariants \cite{khandhawit2010family}.
\end{Prop}

Lipshitz, Ng, and Sarkar, using a filtered refinement of $\psi(L)$ valued in the Lee--Bar-Natan deformation of Khovanov homology, showed that Plamenevskaya's class is invariant under negative flypes \cite{lipshitz2013transverse}.  The above proposition could be seen as evidence that $\kappa$ carries non-classical information even if $\psi$ does not.

$\kappa$ has nice properties mirroring those of $\psi$, and our calculations have some interesting consequences. In Section \ref{section:properties} we collect these observations.  In particular, we show using Proposition \ref{ngexample} that the spectral sequence from $\SKh$ to $\Kh$ does not necessarily collapse immediately, providing a counterexample to Conjecture 4.2 from \cite{hunt2015computing}.  In addition, our work together with that of Baldwin and Grigsby in \cite{baldwin2012categorified} provides a solution (faster than that of \cite{baldwin2012categorified}) to the word problem for braids. 

Recall that the Khovanov chain complex has two reduced variants obtained by placing a basepoint on the link diagram \cite{khovanovreduced}.  The homologies of these complexes are isomorphic as bigraded objects up to a global grading shift.  The behavior of $\kappa$ under positive stabilization provided some promise that a reduced analogue of $\kappa$ might be a transverse invariant. In Section \ref{section:reduced} we define $\kappa$ for both versions of reduced Khovanov homology.  However, these constructions depend on the placement of the basepoint.  We still have some hope that these reduced constructions will provide non-classical transverse information.  In any case, the fact that the two reduced variants are largely independent demonstrates that the two reductions of Khovanov homology are quite different with respect to the $k$-grading.

This project was inspired by similar spectral sequence constructions in Floer homology.  Let $(Y, \xi)$ be a contact three-manifold.  Recall that there are elements $c_\xi \in \widehat{HF}(Y)$ (Heegaard Floer homology) and $\emptyset_\xi \in ECH(Y)$ (embedded contact homology) which are invariants of $\xi$.  It is known that each of these elements vanishes if $(Y,\xi)$ is overtwisted (\cite{contactinvariant}, \cite{echcontactinvariant}) or if $(Y,\xi)$ contains \emph{Giroux $n$-torsion} for any $n > 0$ (\cite{contactvanishes}) (both converses are false).  In \cite{algebraictorsion}, Latschev and Wendl study \emph{algebraic torsion} in symplectic field theory and show that it can obstruct fillability.  Hutchings adapts this work to embedded contact homology by constructing a relative filtration on $ECH(Y)$.  He defines the algebraic torsion of the contact element to be the lowest filtration level at which $\emptyset_\xi$ vanishes.  As $ECH$ is known to be isomorphic to $\widehat{HF}$ (see \cite{echHF}) by an isomorphism carrying $\emptyset_\xi$ to $c_\xi$, it is reasonable to suspect that there is an analogous construction in Heegaard Floer homology .  This is the subject of ongoing work by Baldwin and Vela-Vick and independently by Kutluhan, Mati\'c, Van-Horn Morris, and Wand.  

Now let $L$ be a link with mirror $m(L)$, and let $\Sigma(L)$ be the double cover of $S^3$ branched over $L$.  There is a spectral sequence $E^i(L)$ so that $E^2 \cong \widetilde{CKh}(m(L))$ and $E^\infty = \widehat{HF}(\Sigma(L))$ \cite{OzsvathSzabo}.  If $L$ is a transverse link then $\Sigma(L)$ inherits a contact structure $\xi(L)$.  Plamanevskaya conjectured \cite{plamenevskaya2004transverse} and Roberts proved \cite{robertsSKH} (see also \cite{BaldwinPlam} ) that $\psi(L)$ ``converges'' to $c_{\xi(L)}$ in the sense that there is some $x \in E^0(L)$ so that $[x]_2 = \psi(L)  \in E^2(L)$ and $[x]_\infty = c_{\xi(L)} \in E^\infty(L)$.  This is a weak sort of convergence -- in particular, the vanishing or non-vanishing of the two elements are independent -- but it has been used fruitfully in e.g. \cite{BaldwinPlam}.  We hope to use this connection to derive contact-theoretic information from $\kappa$.

\subsection*{Acknowledgements} We thank John Baldwin for suggesting the project and Eli Grigsby for pointing us towards the annular grading.  We also thank them as well as Olga Plamenevskaya and David Treumann for helpful conversations and advice.

\section{Preliminaries}\label{preliminaries}

\begin{figure}[ht!]
\includegraphics[scale=0.75]{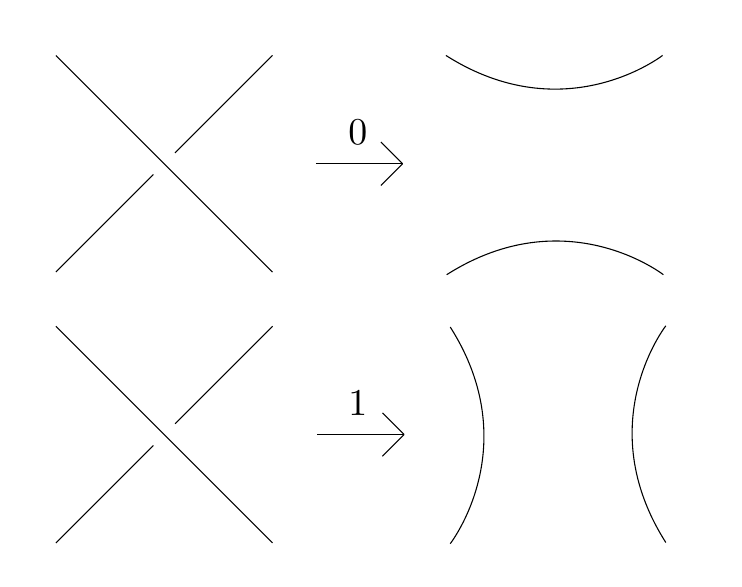}
\caption{Resolutions for the Khovanov chain complex.}
\label{resolutions}
\end{figure}

Let $L \subset S^3$ be a link with diagram $\diagram$.  Suppose that $\diagram$ has $c$ crossings.  Each crossing may be resolved in one of two ways as shown in Figure \ref{resolutions}.  Thus a diagram with $c$ crossings admits $2^c$ complete resolutions, indexed by $\{0,1\}$.\footnote{We assume that the crossings of $\diagram$ are ordered.}  Let $V$ be the vector space generated over $\bz/2\bz$ by the symbols $v_+$ and $v_-$.  For a complete resolution $I$ with $m$ closed components, define $\CKh_I(\diagram) = V^{\otimes m}$.  Concretely, the simple tensors in $\CKh_I(\diagram)$ are labelings of the components of the resolved diagram with $v_+$ and $v_-$.  We refer to these as \emph{canonical generators}.  The underlying vector space for the Khovanov chain complex is $\CKh(\diagram) = \bigoplus_{I \in \{0,1\}^c} \CKh_I(\diagram)$.  The vector space $\CKh(\diagram)$ admits two gradings called $h$ and $q$.  See \cite{khovanov2000} and \cite{bar2005khovanov} for a more complete description of $\CKh$, its gradings, and its differential.  There are two version of \emph{reduced Khovanov homology} which we study in Section \ref{section:reduced}.

Let $A \subset \br^2$ be a standard annulus in $\br^2$.   An \emph{annular link} is a link $L \subset A \times [0,1]$.  Let $\gamma$ be a simple closed curve from the inner boundary of $A \times \{\frac{1}{2}\}$ to the outer boundary.  Let $\pi: A \times I \to A$ be the projection.  Let $L$ be an annular link with diagram $\diagram$.  A component $C$ of the resolved diagram is called \emph{trivial} if the mod 2 intersection number of $\pi(C)$ with $\pi(\gamma)$ is $0$ and is called \emph{non-trivial} otherwise.  The $k$-grading of a generator $x$ is 
\begin{align*}k(x) &= \#\{\text{non-trivial circles in $x$ labeled $v_+$}\} \\ &- \#\{\text{non-trivial circles in $x$ labeled $v_-$}\}.\end{align*}
Roberts \cite{robertsSKH}, following \cite{apsSKH}, shows that the Khovanov differential is non-increasing in $k$.  Thus the subcomplexes $\filt_i(\diagram) = \{x \in \CKh(\diagram) : k(x) \leq i\}$ form a bounded filtration of $\CKh(L)$.  Moreover, the filtered chain homotopy type of $\CKh(\diagram)$ is an invariant of $L$ as an annular link.  For a filtered complex $(X',d',\filt'_i)$ the \emph{associated graded object} is the direct sum of complexes $\bigoplus_i \filt'_i/\filt'_{i-1}$.  There is a spectral sequence from the associated graded object to the homology of the total complex, see \cite{Mccleary}.  The associated graded object of the Khovanov chain complex filtered by $k$ is called \emph{annular Khovanov homology} and is denoted by $\SKh(L)$.  Roberts concludes the following.

\begin{Thm} \cite{robertsSKH} For any annular link $L$ there is a spectral sequence from $\SKh(L)$ to $Kh(L)$.\end{Thm}   

Braid closures may be naturally regarded as annular links, and annular Khovanov homology has proven to be a powerful tool in studying braids.  See, for example, \cite{grigsby2011gradings}, \cite{baldwin2012categorified}, \cite{grigsby2013sutured}, and \cite{hubbard2015sutured}.

Braid closures are also closely related to transverse links.  We will say that a link $L \subset \br^3$ is \emph{transverse} if it is everywhere transverse to the standard radial contact structure $\xi = \ker(dz + r^2d\theta)$.  (For convenience we state these results for $\br^3$ but they all extend to $S^3$.)  Two links are \emph{transversely isotopic} if they are isotopic through transverse links.  A link which is braided around the $z$-axis may be isotoped to a transverse link, and Bennequin \cite{Bennequin} shows that every transverse link may be isotoped to such a braid closure.   The Transverse Markov Theorem of \cite{OrevShev} and \cite{Wrinkle} states that if $\beta$ and $\beta'$ are braids with transverse closures $\bar\beta$ and $\bar\beta'$ then $\bar\beta$ and $\bar\beta'$ are transversely isotopic if and only if $\beta$ and $\beta'$ are related by a sequence of conjugations, positive stabilizations, and positive destabilizations.  Positive stabilization is the operation which sends $\beta \in B_n$ to $\beta\sigma_n \in B_{n+1}$, and destabilization is its inverse.  There are two classical invariants of transverse links: the underlying smooth link type and the \emph{self-linking number}.  The self-linking number of $\bar\beta$ is \[ sl(\bar\beta) = a(\beta) - n \]  where $a(\beta)$ is the sum of the exponents of $\beta$.  An invariant of transverse links is called \emph{effective} if it separates transverse links in the same smooth isotopy class with the same self-linking number.  For more on the connection between transverse links and braids, see \cite{Etnyre}.

Now let $\beta$ be an $n$-strand braid whose closure $\bar\beta$ is the transverse link $L$.  Let $\diagram$ be a diagram for $\bar\beta$.  Recall that $\psi(\diagram)$ is defined as the generator in the braidlike resolution with only $v_-$ labels.  An easy calculation shows that $q(\psi(\diagram)) = sl(L)$.  Suppose that $\beta'$ is another braid whose closure is transversely isotopic to $L$.  Then for any sequence of conjugations and positive (de)stabilizations that transforms $\beta$ into $\beta'$, the naturally induced map $\CKh(\bar\beta) \to \CKh(\bar\beta')$ carries $\psi(\bar\beta)$ to $\psi(\bar\beta')$.  Thus $\psi(L) \in \CKh(L)$ is well-defined.  Note that while the $k$-grading of $\psi(L)$ is not well-defined (it is decreased by positive stabilization, see Lemma \ref{posstab}), the element $\psi(\diagram)$ generates $\filt_{-n}$, the lowest non-trivial level of the $k$-filtration.  It is easy to show that $\psi(L)$ is a cycle.

It is not known if $\psi(L)$ is effective.  Plamenevskaya showed that $[\psi(\bar\beta)] \neq 0$ if $\beta$ is negatively destabilizable, if $\beta$ can be written with $\sigma_i^{-1}$ but not $\sigma_i$ \cite{plamenevskaya2004transverse}, or if $\beta$ is non-right-veering (with Baldwin, \cite{baldwin2012categorified}).  On the other hand, if $\beta$ is quasi-positive then $[\psi(\beta)] \neq 0$ \cite{plamenevskaya2004transverse}.

\section{Definition and invariance of $\kappa$} \label{section:definitionandinvariance}

Let $B_n$ be the braid group on $n$ strands and let $\beta \in B_n$ be a braid with transverse element $\psi(\bar\beta)$.  The $k$-filtration on $\CKh(\bar\beta)$ has the form
\[0 \subset \filt_{-n} \subset \filt_{2-n} \subset \cdots \subset \filt_{n-2} \subset \filt_n = \CKh(\bar{\beta}) \] where $\filt_{-n}$ is generated by $\psi(\bar\beta)$, so $\psi(\bar{\beta}) \in \filt_i$ for $i \geq -n$.  We restate Definition \ref{def:kappa}:

\begin{Def*} Let $\beta \in B_n$ and suppose that $\psi(\bar\beta)$ is a boundary in $\CKh(\bar\beta)$.  Define \[ \kappa(\beta) = n + \min \{i :  [\psi(\bar\beta)] = 0 \in H(\filt_i)\}. \]  If $\psi(\bar\beta)$ is not a boundary, then define $\kappa(\beta) = \infty$.\end{Def*}

We will say that $y \in \CKh(\bar\beta)$ \emph{realizes $\kappa(\beta)$} if $dy = \psi(\bar\beta)$ and $k(y) = \kappa(\beta) - n$.  Note that $\kappa$ is always even and that $2 \leq \kappa(\beta) \leq 2n$.  The only element with $k$-grading $n$ is the all $v_+$ labeling of the braidlike resolution, so in fact $\kappa(\beta) \leq 2(n-1)$.  We now show that $\kappa$ is a well-defined function on $B_n$.  First, an algebraic lemma.

\begin{Lem}\label{lem:algebra} Let $(X, d, \filt)$ and $(X', d', \filt')$ be complexes with bounded filtrations, and suppose that $f: X \to X'$ is a filtered chain map.  For any non-zero cycle $x \in X$, define $\kappa(x) = \min\{i : [x] =0 \in H_*(\filt_i)\}$ or $\kappa(x) = \infty$ if $x$ is not a boundary.  Define $\kappa'$ analogously on $X'$.  Suppose that $f(x) = y \neq 0$.  Then $\kappa(x) \geq \kappa'(y)$.  If there is a filtered chain map $g: X' \to X$ with $g(y) = x$, then $\kappa(x) = \kappa'(y)$.\end{Lem}

\begin{proof} Chain maps carry cycles to cycles, so if $\kappa(x)$ is defined then so is $\kappa'(y)$.  There is nothing left to prove if $\kappa(x) = \infty$, so suppose that $\kappa(x)$ is finite.  Then there is some $w \in \filt_{\kappa(x)}$ so that $dw = x$, and $(f \circ d)(w) = y = (d \circ f)(w)$.  As $f$ is filtered, $f(w) \in \filt'_{\kappa(x)}$, so $\kappa'(y) \leq \kappa(x)$.  If there is a filtered chain map $g$ with $g(y) = x$, then the opposite inequality shows that $\kappa(x) = \kappa'(y)$.
\end{proof}

The $\kappa$ of Lemma \ref{lem:algebra} differs from that of Definition \ref{def:kappa} in that the latter is normalized using the braid index, but the Lemma clearly still applies to the Definition.

\begin{Prop} Suppose that $\beta$ and $\beta'$ are words in the Artin generators so that $\beta = \beta'$ in $B_n$.  Then $\kappa(\beta) = \kappa(\beta')$.  \end{Prop}
\begin{proof}
It will suffice to show that $\kappa(\beta)$ is invariant under Reidemeister 2 and Reidemeister 3 moves which do not cross the braid axis.  These moves induce natural maps on the Khovanov chain complex which carry $\psi(\beta)$ to $\psi(\beta')$, see \cite{plamenevskaya2004transverse}.  For a digestible summary of these maps, see \cite{bar2005khovanov}.  If these maps are filtered, then Lemma \ref{lem:algebra} completes the proof. 

The map induced by Reidemeister 2 (and its inverse) is a direct sum of identity maps and compositions of saddles with cups and caps.  The saddles, cups, and caps do not cross the braid axis.  Certainly the identity map is filtered.  One may check directly that saddle maps are filtered; alternatively, observe that a saddle may be viewed as a component of the Khovanov differential of some annular link and so it must be filtered.  Cups and caps that do not cross the braid axis cannot change the $k$-grading.  Thus the Reidemeister 2 map is filtered.  An identical analysis shows that the Reidemeister 3 maps are filtered.
\end{proof}

Considering braids instead of their closures, we obtain the following.
\begin{Prop} 
$\kappa$ is an invariant of conjugacy classes in $B_n$.
\end{Prop}

A program to compute $\kappa$ is available at \url{www2.bc.edu/adam-r-saltz/kappa.html}.

\section{Examples and Properties of $\kappa$}\label{section:properties}

\subsection{Main example}

An immediate first question is whether elements in $k$-grading $-n+2$ always suffice to kill $\psi(\overline\beta)$ whenever $[\psi(\overline\beta)] = 0$, that is, whether $\kappa = 2$ for all braids with $[\psi(\overline\beta)] = 0$.  Proposition \ref{ngexample}, using examples from Theorem 1.1 in \cite{khandhawit2010family}, shows that this is false.  We restate it here for convenience:

\begin{Prop*} For any $a,b \in \{0,1,2\}$, the pair of closed $4$-braids
\[ A(a,b) = \sigma_{3}\sigma_{2}^{-2}\sigma_{3}^{2a+2}\sigma_{2}\sigma_{3}^{-1}\sigma_{1}^{-1}\sigma_{2}\sigma_{1}^{2b+2} \,\,\,\,\,\text{and}\,\,\, \]
\[B(a,b) = \sigma_{3}\sigma_{2}^{-2}\sigma_{3}^{2a+2}\sigma_{2}\sigma_{3}^{-1}\sigma_{1}^{2b+2}\sigma_{2}\sigma_{1}^{-1},\]
related by a negative flype, can be distinguished by $\kappa$: indeed, $\kappa(A(a,b0)) = 4$ and $\kappa(B(a,b)) = 2$. 
\end{Prop*}
\begin{proof} By computation.
\end{proof}

We do not know if this relation holds for all $a, b \in \bz_{\geq 0}$. Recall that since $\overline A(a,b)$ and $\overline B(a,b)$ are in the same isotopy class (as they are related by a flype), they have isomorphic Khovanov homologies.  However, annular Khovanov homology can differentiate them (see \cite{hubbard2015sutured}) for $a, b \in \{0,1,2\}$.

\subsection{Negative Stabilization}

\begin{Prop}\label{neg} If a closed $n$-braid $\beta$ is a negative stabilization of another braid, then $\kappa(\beta) = 2$.
\end{Prop}

\begin{proof} In Theorem 3 of \cite{plamenevskaya2004transverse}, Plamenevskaya constructs an element $y \in CKh(\beta)$ such that $d y = \psi(\beta)$ as follows: consider the resolution formed from taking the $0$-resolution of the negative crossing from the negative stabilization, the $1$-resolution for all other negative crossings, and the $0$-resolution for all positive crossings.  The element $y$ is obtained by assigning each circle in this resolution $v_{-}$.  It is clear that $y$ has $k$-grading $-n+2$.
\end{proof}

\subsection{Positive Stabilization}\label{posstab}

Define an $\textit{arc}$ of a closed braid diagram to be a segment of the link that goes from one crossing to another crossing without traversing over or under any other crossings.  An $\textit{innermost arc}$ is one for which we can draw a straight line from the braid axis to any point on the arc without crossing any other arcs.  An $\textit{innermost point}$ is a point lying on an innermost arc.

Given an $n$-strand braid $\beta$, we define $S^{p}\beta$ to be $\beta$ positively stabilized once at an innermost point $p$.  That is: insert $\sigma_{n}$ at the point $p$ on the diagram.   

\begin{Prop} $\kappa(\beta)$ is not a transverse invariant.
\end{Prop}

\begin{proof}

This is due to the fact that the chain map corresponding to positive stabilization is not filtered (see Proposition \ref{boundspos}). We have a concrete example: consider the braid $B(0,0)$ from Proposition \ref{ngexample}. By computation, $\kappa(B(0,0))=2$ and $\kappa(S^{p}B(0,0)) = 4$ for all choices of innermost points $p$. 
\end{proof}

We note here that we can define a transverse invariant using $\kappa$, though it is not clear how to compute it unless the transverse link is known to be represented by a braid with $\kappa = 2$.

\begin{Def} For an $n$-braid $\beta$, define $\kappa_{min}(\beta)$ to be the minimum $\kappa(K)$ over all transverse representatives $K$ of $\beta$.  It is a transverse invariant.
\end{Def}

We can give bounds on the behavior of $\kappa$ under positive stabilization:

\begin{Prop}\label{boundspos} $\kappa(\beta) \leq \kappa(S^{p}\beta) \leq \kappa(\beta) + 2$. 
\end{Prop}

\begin{figure}[ht!]
\includegraphics[scale=0.8]{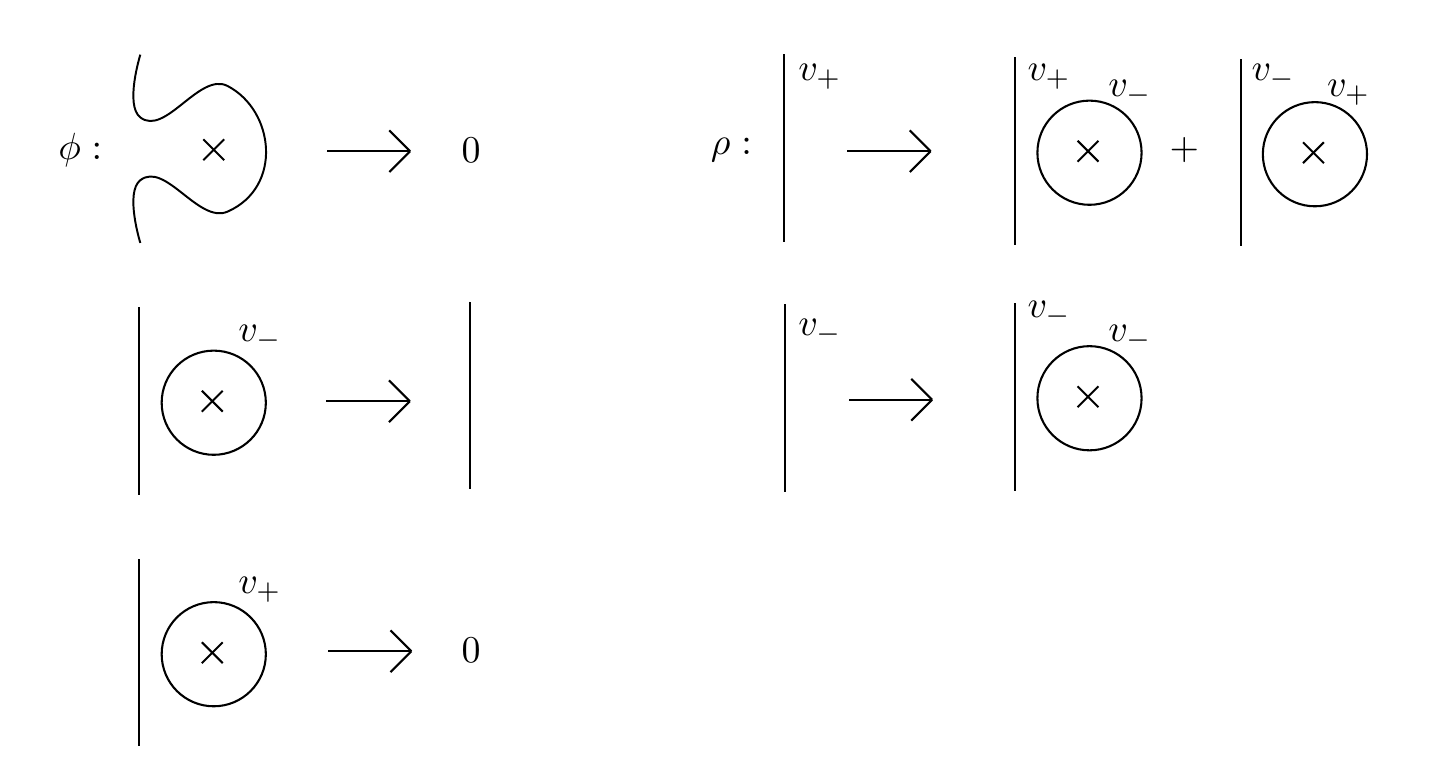}
\caption{Chain maps for positive stabilization}
\label{posstab_fig}
\end{figure}

\begin{proof} 
$S^{p}\beta$ has a positive crossing at $p$, and for an $n$-strand braid $\beta$ we refer to this crossing as $\sigma_{n,p}$.  Suppose that $\sigma_{n,p}$ appears last in the crossing ordering. We show the first inequality.  As described in \cite{bar2005khovanov}, there is a chain map $\phi: CKh(S^{p}\beta) \to CKh(\beta)$ sending all elements in resolutions of $S^{p}\beta$ where $\sigma_{n,p}$ is $1$-resolved to $0$ and satisfying 
\[ \phi(z \otimes v_{-}) = z \]
\[ \phi(z \otimes v_{+}) = 0 \] 
for elements in resolutions where $\sigma_{n,p}$ is $0$-resolved (see Figure \ref{posstab_fig}).  Consider an element $y \in CKh(S^{p}\beta)$ realizing $\kappa(S^{p}\beta)$.  The element $y$ takes the form $z_{1} \otimes v_{-} + z_{2} \otimes v_{+} + z_{3}$.  So
\[ d(\phi(y)) = d(z_{1}) = \phi(d y) = \phi(\psi(S^{p}\beta)) = \psi(\beta) \]
Hence $z_{1}$ kills $\psi(\beta)$, and so we have 
\[ \kappa(S^{p}\beta) = \max k(z_{1} \otimes v_{-}, z_{2} \otimes v_{+}, z_{3}) + n + 1 \] \[\geq k(z_{1} \otimes v_{-}) + n + 1 = k(z_{1}) + n \geq \kappa(\beta) \]

As described in \cite{bar2005khovanov} (see also \cite{plamenevskaya2004transverse}), there is a chain map $\rho: CKh(\beta) \to CKh(S^{p}\beta)$ satisfying $\rho(\psi(\beta)) = \psi(S^{p}\beta)$.  It is given by
\[\rho(v_{-}) = v_{-} \otimes v_{-}\]
\[\rho(v_{+}) = v_{+} \otimes v_{-} + v_{-} \otimes v_{+}\]
Hence $\rho$ can either decrease $k$-grading by one or increase it by one, depending on whether the circles in question are trivial or non-trivial.  
Now, suppose we have an element $y \in CKh(\beta)$ realizing $\kappa(\beta)$: then $\rho(y)$ kills $\psi(S^{p}\beta)$.  The $k$-grading of $\rho(y)$ is at most $\kappa(\beta) -n + 1$.  Stabilization increases strand number by one, so $\kappa(S^{p}\beta)$ could at most be 
\[\kappa(\beta) -n + 1 + n + 1 = \kappa(\beta) + 2.\]
\end{proof}

\subsection{Other properties and consequences}

Propositions \ref{neg} and \ref{boundspos} immediately give us bounds for $\kappa$ of braids related by exchange moves and positive flypes: 

\begin{Prop} If two braids $\sigma$ and $\beta$ are related by a single exchange move or a single positive flype, then $|\kappa(\sigma) - \kappa(\beta)| \leq 2$.
\end{Prop}

\begin{proof} Exchange moves and positive flypes can both be expressed as a composition of braid isotopies, one single positive stabilization, and one single positive destabilization (see for instance \cite{birman2000transversally}, \cite{lipshitz2013transverse}).
\end{proof}


\begin{Prop}\label{sigmanegative} Suppose a closed $n$-braid $\beta$ can be represented by a braid word containing a factor of $\sigma_{i}^{-1}$ but no $\sigma_{i}$'s for some $i = 1, \ldots, n$.  Then $\kappa(\beta) = 2$.
\end{Prop}

\begin{proof}
The argument we give here is very similar to arguments found in \cite{plamenevskaya2004transverse}. Consider the resolution formed from taking the $0$-resolution of one of the $\sigma_{i}^{-1}$'s, the $1$-resolution for all other negative crossings, and the $0$-resolution for all positive crossings.  We claim that assigning each circle in this resolution $v_{-}$ yields an element $y$ with $d y = \psi$ and $k(y) = -n+2$.  The differential $d$ on $y$ is the sum of all maps with $y$ as their initial end.  By our choice of resolution, any map corresponding to a merge map sends $y$ to $0$.  Hence $d$ is a sum of split maps.  Topologically, the only split maps that can start from this resolution are in the $i$'th column; however, there are only negative crossings in this column, and at this resolution they are all $1$-resolved except for the one that is $0$-resolved.  So the only contributor to $d y$ is the map resolving that crossing, sending $y$ to $\psi(\beta)$.  
\end{proof}

The following two definitions (with more detail) can be found in \cite{baldwin2012categorified}.  Let $D_{n}$ denote the standard unit disk with $n$ marked points $p_{1}, \ldots, p_{n}$ positioned along the real axis.

\begin{Def} An arc $\gamma: [0,1] \to D_{n}$ is admissible if it satisfies
\begin{enumerate}[(i)]
\item $\gamma$ is a smooth imbedding transverse to $\partial D_{n}$
\item $\gamma(0) = -1 \in \mathbb{C}$ and $\gamma(1) \in \{p_{1}, \ldots, p_{n}\}$
\item $\gamma(t) \in D_{n} - (\partial D_{n} \cup \{p_{1}, \ldots, p_{n}\})$ for all $t \in (0,1)$ and
\item $\frac{d\gamma}{dt} \neq 0$ for all $t \in (0,1)$.
\end{enumerate}
\end{Def}

\begin{Def} Let $\sigma \in B_{n}$. We say $\sigma$ is right-veering if for all admissible arcs $\gamma$, $\sigma(\gamma)$ is right of $\gamma$ when pulled tight.
\end{Def}

\begin{Cor}\label{rightveering} If an $n$-braid $\sigma$ is not right-veering, then $\kappa(\overline{\sigma}) = 2$. 
\end{Cor}

\begin{proof}
By Proposition 3.1 of \cite{baldwin2012categorified} and Proposition 6.2.7 of \cite{dehornoy2002braids}, $\sigma$ is conjugate to a braid that can be represented by a word containing at least a factor of $\sigma_{i}^{-1}$ but no $\sigma_{i}$'s for some $i = 1, \ldots, n$. The result follows by Proposition \ref{sigmanegative}.
\end{proof}

For a braid $\beta \in B_{n}$ we denote its mirror as $m(\beta) \in B_{n}$.

\begin{Cor}\label{word} If $\kappa(\overline{\sigma}) \neq 2$ and $\kappa(m(\overline{\sigma})) \neq 2$, then $\sigma = 1 \in B_{n}$.
\end{Cor}

\begin{proof}
The proof is similar to that of Corollary 1 of \cite{baldwin2012categorified}.  By Corollary \ref{rightveering}, $\sigma$ and $m(\sigma)$ are right-veering and hence $\sigma$ is also left-veering. By Lemma 3.1 of \cite{baldwin2012categorified}, $\sigma$ is the identity braid.
\end{proof}

We note that this provides a solution to the word problem (indeed, faster than that of \cite{baldwin2012categorified}), since it is only necessary to check if Plamenevskaya's invariant vanishes by the $E^{3}$ page of the spectral sequence from annular Khovanov homology to Khovanov homology.

$\kappa$ provides an obstruction to negative destabilization (Proposition \ref{neg}).  It can also provide an obstruction to positive destabilization for a braid in the case that $\kappa \neq 2$ for its mirror.  Corollary \ref{word} implies that it cannot provide an obstruction to destabilization in general.  One might hope to show that $\kappa \neq 2$ for a braid and $\kappa \neq 2$ for its mirror, implying that the braid is neither negatively destabilizable nor positively destabilizable.  However, Corollary \ref{word} shows that if this is the case, the braid is trivial.

We end this section with a remark on spectral sequences. For any annular link $\mathbb{L}$, there is a spectral sequence whose $E^{0}$ page is the Khovanov complex of $\mathbb{L}$, $CKh(\mathbb{L})$ and whose $E^{1}$ page is, as a group, the annular Khovanov homology of $\mathbb{L}$.  Since there are no differentials that drop the $k$-grading by one, the $E^{2}$ page is identical to the $E^{1}$ page. Therefore the first page at which the spectral sequence could collapse is $E^{3}$. The following proposition provides a counterexample to Conjecture 4.2 from \cite{hunt2015computing}.

\begin{Prop}\label{spectral} The spectral sequence from annular Khovanov homology to Khovanov homology does not always collapse at the $E^{3}$ page.  
\end{Prop}

\begin{proof} We consider the braid $A(0,0)$ from Proposition \ref{ngexample}.  The distinguished element $\psi(A(0,0))$ lives in homological grading 4 (before any final shifts) and has $k$ grading $-4$.  Recall that $\psi(A(0,0))$ is unique in the lowest $k$-grading. By $E^{r}_{d,m}$ we mean the $r$'th page of the spectral sequence at homological grading $d$ and $k$-grading $m$. 

Following \cite{hutchings2011introduction} (recall: the differentials on $CKh$ increase homological grading),
\[ E^{3}_{4,-4} = \frac{\{x \in \mathcal{F}_{-4}CKh_{4} | d x \in \mathcal{F}_{-7}CKH_{5} \}}{\mathcal{F}_{-5}CKh_{4} + d(\mathcal{F}_{-2}CKh_{3})} = \frac{\{x \in \mathcal{F}_{-4}CKh_{4} | d x = 0 \}}{d(\mathcal{F}_{-2}CKh_{3})} \]
\[ =\frac{\psi(A(0,0))}{d(\mathcal{F}_{-2}CKh_{3})} = [\psi(A(0,0))]  \neq 0 \]
since $\kappa(A(0,0)) \neq 2$. However, $[\psi(A(0,0))] = 0 \in Kh_{4}(A)$, and hence $Kh_{4}(A(0,0)) \neq \bigoplus_{k=-4}^{4} E^{3}_{4,k}$.  

\end{proof}

Precisely the same argument yields a more general statement:

\begin{Prop}\label{boundspectral} Given a braid $\beta$, the length of the spectral sequence from $\SKh(\overline{\beta})$ to $\Kh(\overline{\beta})$ is bounded below by $\kappa(\beta)$.  
\end{Prop}

\section{Invariants in reduced Khovanov homology}\label{section:reduced}

It is implicit in the proof of Proposition \ref{boundspos} that $\kappa$ increases under positive stabilization at $p$ precisely if every element which realizes $\kappa$ has a canonical summand in which $p$ lies on a trivial $v_+$-labeled circle.  This situation cannot occur in (one version of) reduced Khovanov homology, and so one might hope that a ``reduced $\kappa$'' is an invariant of transverse links.  That's not quite the case -- the eager reader may skip to the examples at the end of this section -- but the reduced invariants are interesting in their own right.

In this section let $p$ be a non-double point on an $n$-strand annular braid diagram $\diagram$ of $\bar\beta$.  For convenience, we will assume that the last tensor factor of each generator of $\CKh(\diagram)$ corresponds to the component containing $p$.   There is a chain map $x_p: \CKh(\diagram) \to \CKh(\diagram)$ defined on generators by 
\begin{align*} 
x_p (y \otimes v_+) & = y \otimes v_-\\
x_p (y \otimes v_-) &= 0
\end{align*}
There are two flavors of \emph{reduced Khovanov homology.}  The \emph{reduced subcomplex} $\reducedsub_p(\diagram)$ is defined as $\ker(x_p)$.  The \emph{reduced quotient complex} $\reducedquo_p(\diagram)$ is defined as $\coker(x_p)$.  The homologies of these complexes are both called reduced Khovanov homology; the ambiguity is justified by the fact that their homologies are isomorphic as $h$- and $q$-graded complexes (with a constant shift in the $q$-gradings).  It is clear that $\reducedsub_p(\diagram)$ has a basis of canonical generators.  The projections of canonical generators of the form $y \otimes v_+$ form a basis of $\reducedquo_p(\diagram)$.  Whenever we take a representative of an element in the quotient complex, we will assume it is a sum of these canonical generators.

The $k$-grading on $\CKh(\diagram)$ induces a $k$-grading on each variant.  On the subcomplex $\reducedsub_p(\diagram)$ this is simply the restriction.  We define the $k$-grading on $\reducedquo_p(\diagram)$ via canonical representatives: if $y$ is the canonical representative of $\ul{y} \in \reducedquo_p(\diagram)$, then $k(\ul{y}) = k(y)$.  However, the isomorphism between the two variants is not in general $k$-filtered.  Thus we will distinguish their homologies as the \emph{reduced homology} $\widetilde{Kh}_p(\diagram)$ and the \emph{reduced quotient homology} $\underline{Kh}_p(\diagram)$.  We write $\widetilde{\filt}_i$ and $\underline{\filt}_i$ for the $i$th filtered levels of $\reducedsub_p(\diagram)$ and $\reducedquo_p(\diagram)$ respectively.

Each complex supports a variant of the transverse element $\psi(\diagram)$.  The cycle corresponding to $\psi(\diagram)$ is also a cycle in the subcomplex $\reducedsub_p(\diagram)$ for any $p$.  When we wish to emphasize that we are considering $\psi(\diagram)$ as an element of the subcomplex, we will write it as $\subpsip(\diagram)$.  Plamenevskaya defines the reduced quotient invariant $\quotpsip(\diagram)$ to be the image of the chain $v_- \otimes \cdots \otimes v_- \otimes v_+$ in $\reducedquo_p(\diagram)$.  Both $\subpsip$ and $\quotpsip$ are invariant under braid conjugation and stabilization away from $p$ in the same sense (and with the same proofs) as $\psi$.  Both cycles have the lowest $k$-grading in their respective complexes, but $\quotpsip$ does not necessarily generate that lowest level.

Note that these constructions depend on a choice of $p$ on a particular diagram for a link, so we will not write ``$\subpsip(\bar\beta)$'' or ``$\quotpsip(\bar\beta)$''.  

\begin{Def} Let $\beta \in B_n$, let $\diagram$ be an annular diagram for $\bar\beta$, and let $p \in \diagram$.  If $\subpsip(\diagram)$ is a boundary in $\reducedsub_p(\diagram)$, define \[ \subkapp(\diagram) = n + \min \{i :  [\subpsip(\diagram)] = 0 \in H_*(\widetilde{\filt}_i(\diagram))\}. \]  If $\subpsip(\diagram)$ is not a boundary, then define $\subkapp(\diagram) = \infty$.  If $\psi_p(\diagram)$ is a boundary in $\reducedquo_p(\diagram)$, define \[ \quotkapp(\diagram) = n + \min \{i :  [\quotpsip(\diagram)] = 0 \in H_*(\ul{\filt}_i(\diagram))\}. \]  If $\quotpsip(\diagram)$ is not a boundary, then define $\quotkapp(\diagram) = \infty$.\end{Def}

The arguments of section \ref{section:definitionandinvariance} show that $\subkapp(\diagram)$ and $\quotkapp(\diagram)$ are invariant under positive stabilization away from $p$ and conjugations that do not cross $p$.

\begin{Lem} For a fixed diagram $\diagram$, either $\kappa(\diagram)$, $\subkapp(\diagram)$, and $\quotkapp(\diagram)$ are all finite or all infinite.  In the finite case, $\kappa(\diagram) \leq \subkapp(\diagram) \leq \quotkapp(\diagram) \leq \subkapp(\diagram) + 2$.\end{Lem}

\begin{proof} There is a short exact sequence of complexes \[ 0 \to \reducedsub_p(\diagram) \xrightarrow{i} \CKh(\diagram) \xrightarrow{\pi} \reducedquo_p(\diagram) \to 0 \] where $i$ is the inclusion and $\pi$ is the projection to the quotient.  The induced map on homology $i_*$ carries $[\subpsip(\diagram)]$ to $[\psi(\diagram)]$, so if $[\psi(\diagram)] \neq 0$ then $[\subpsip(\diagram)] \neq 0$.  The connecting homomorphism in the long exact sequence on homology is zero so $i_*$ is injective.  This shows if $[\subpsip(\diagram)] \neq 0$, then $[\psi(\diagram)] \neq 0$.

The first piece of the inequality follows immediately from the fact that $\reducedsub_p(\diagram)$ is a subcomplex of $\CKh(\diagram)$.  For the next part, suppose that $\ul{z}$ realizes $\quotkapp(\diagram)$.  Then $d(x_pz) = \psi(\diagram)$ and $k(x_pz) \leq k(\ul{z})$, so $\subkapp(\diagram) \leq \quotkapp(\diagram)$.  On the other hand, suppose that $y$ realizes $\subkapp(\diagram)$; every canonical summand of $y$ has a $v_-$ at $p$.  Let $y^+$ be the element obtained from $y$ by changing those $v_-$'s to $v_+$'s.  Clearly $x_pd(\ul{y}^+) = \psi(\diagram)$, so \[ dy^+ = \quotpsip(\diagram) + \text{terms with $v_-$ at $p$}. \]  Therefore $d\underline{y}^+ = \quotpsip(\diagram)$ and $\quotkapp(\diagram) \leq k(y^+) + n \leq k(y) + 2 + n = \subkapp(\diagram) + 2.$  (This also shows that $\subkapp(\diagram)$ is finite if and only if $\quotkapp(\diagram)$ is finite.)  
\end{proof}

The reduced invariants are stable under positive stabilization at $p$ in the following sense: let $p'$ be a point on the same arc as $p$.  For each reduced complex, the positive stabilization map is filtered and preserves Plamanevskaya's invariant, so Lemma \ref{lem:algebra} implies that the appropriate version of $\kappa$ does not change.  But after this operation the image of $p$ is not an innermost point.  We instead study $\vec{S}^{p'}$, the operation of stabilizing at $p'$ and then moving the basepoint to some point $q$ on the new innermost strand.

\begin{Prop} 
Let $\diagram$ be a diagram of $\bar\beta$.  Then $\tilde{\kappa}_q(\vec{S}^{p'}\diagram) \leq \subkapp(\diagram)$ and $\ul{\kappa}_q(\vec{S}^{p'}\diagram) \leq \quotkapp(\diagram) + 2$.
\end{Prop}
\begin{proof} The first inequality follows from Lemma \ref{lem:algebra} once one makes the observation that the positive stabilization map carries $\reducedsub_p(\diagram)$ to a subcomplex of $\reducedsub_q(\vec{S}^{p'}\diagram)$ and carries $\subpsip(\diagram)$ to $\widetilde{\psi}_q(\diagram)$.

Suppose that $\ul{z}$ realizes $\quotkapp(\diagram)$.  Let $q$ be a point on the innermost strand of $S^{p'}\diagram$.  Note that any canonical generator of $\ul{\CKh}(S^{p'}\diagram)$ has $v_-$ at $q$ and $v_+$ at $p$.  Let $\ul{z}' \in \ul{\CKh}_q(S^{p'}\diagram)$ be the element whose canonical representative is obtained from $S^{p'}\underline{z}$ by swapping these labels.  Note that $dx_q\underline{z}' = S^{p'}\psi(\diagram) = x_qd\underline{z}'$, so $d\underline{z}' = \ul{\psi}_q(\diagram)$.  Clearly $k(\ul{z}') \leq k(\underline{z}) + 2$.
\end{proof}

\begin{Rem*}
It is interesting to consider the sharpness of these inequalities using annular Khovanov homology.  The map $x_p$ is filtered and therefore induces a map on $\SKh(\diagram) = \bigoplus \filt_i/\filt_{i-1}$, the annular Khovanov homology of $\diagram$. 

Let $p$, $p'$, $z$, and $z'$ be as in the previous proof.  The point $q$ lies on a non-trivial circle in every resolution of $S^{p'}\diagram$.  Thus $k(z') > k(S^{p'}z) = k(z)$ precisely if $p$ lies on a trivial circle in some canonical summand of $z$.  Equivalently, $k(z') = k(z)$ precisely if $p$ lies on a non-trivial circle in every canonical summand of $z$.  Thus $\quotkapp(\diagram) = \quotkapp(\vec{S}^{p'})$ if and only if some $z$ realizes $\quotkapp(\diagram)$ and  $p$ lies on a non-trivial circle in every canonical summand of $z$.  Write $\langle z \rangle$ for the image of $z$ in $\SKh(\diagram)$.  Then $\quotkapp(\diagram) = \quotkapp(\vec{S}^{p'}\diagram)$ if and only if $\langle z \rangle \in \ker(x_p)$ for some $z$ which realizes $\quotkapp(\diagram)$.

The map $x_p$ induces maps on $\bigoplus \filt_i/\filt_{i-1}$, $\bigoplus \ul\filt_i/\ul\filt_{i-1}$, and $\bigoplus \tilde\filt_i/\tilde\filt_{i-1}$.  Recall that these are the annular Khovanov chain group and the two reduced annular Khovanov chain groups, respectively.  In this language, $k(z') = k(z)$ -- and therefore $\quotkapp(\diagram) = \ul{\kappa}_q(\vec{S}^{p'}\diagram)$ -- if and only if the image of $z$ in $\SKh(\diagram)$ lies in the kernel of $x_p$.
\end{Rem*}

\begin{figure}
\centering
  \includegraphics{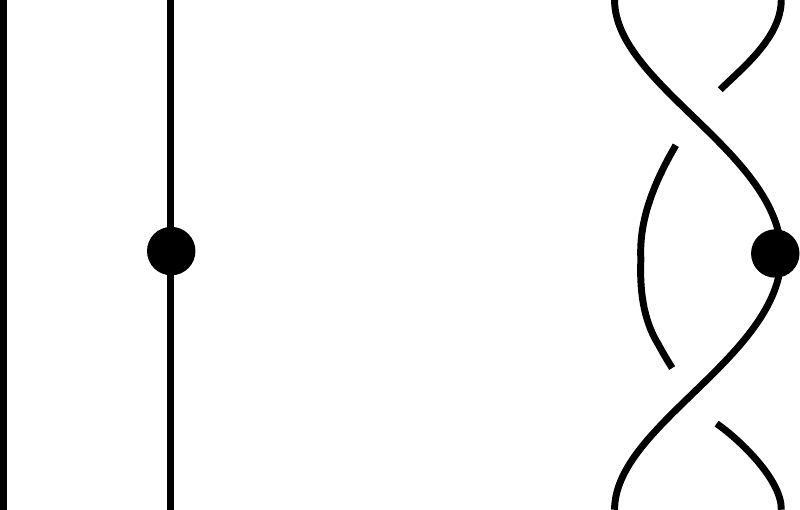}
  \caption{The result of the operation $\vec{C}_p$ on two strands.}
\label{fig:conjugation}
\end{figure}

While $\quotkapp$ is not preserved under stabilization, it is preserved under a certain sort of conjugation over $p$.  Denote by $C_p$ the operation of performing a braidlike Reidemeister 2 move over $p$.  (In terms of braid words, this inserts $\sigma_{n-1}\sigma_{n-1}^{-1}$ or $\sigma_{n-1}^{-1}\sigma_{n-1}$.)  Denote by $\vec{C}_p$ the operation $C_p$ followed by moving the basepoint to the innermost strand at $q$.  See Figure \ref{fig:conjugation}.  The Reidemeister 2 map induces a filtered map $\reducedquo_p(\diagram) \to \reducedquo_q(\vec{C}_p\diagram)$ which carries $\psi$ to $\psi$.
\begin{Prop}
$\quotkapp(\diagram) = \ul\kappa_q(\vec{C}_p\diagram).$
\end{Prop}

To dash any hope that $\quotkapp$ or $\subkapp$ might be transverse invariants, we note that both invariants depend on $p$.  For $\quotkapp$ this is true even for negative stabilizations.  

\begin{Ex*}
Let $\beta = \sigma_1\sigma_2^{-1} \in B_3$.  Certainly $\psi$ is null-homologous and $\kappa = \subkapp = 2$ for any $p$.  Let $p_1$ and $p_2$ be points on the first and second strands of the braid.  Then
\begin{align*}
 \ul\kappa_{p_1} &= 2\\
 \ul\kappa_{p_2} &= 4
\end{align*}
\end{Ex*}
For a meatier example, we revisit the transversely non-simple family using the previously advertised computer program.

\begin{figure}[ht!]
\hspace*{-0.9in}
	\centerline{\includegraphics[scale=1]{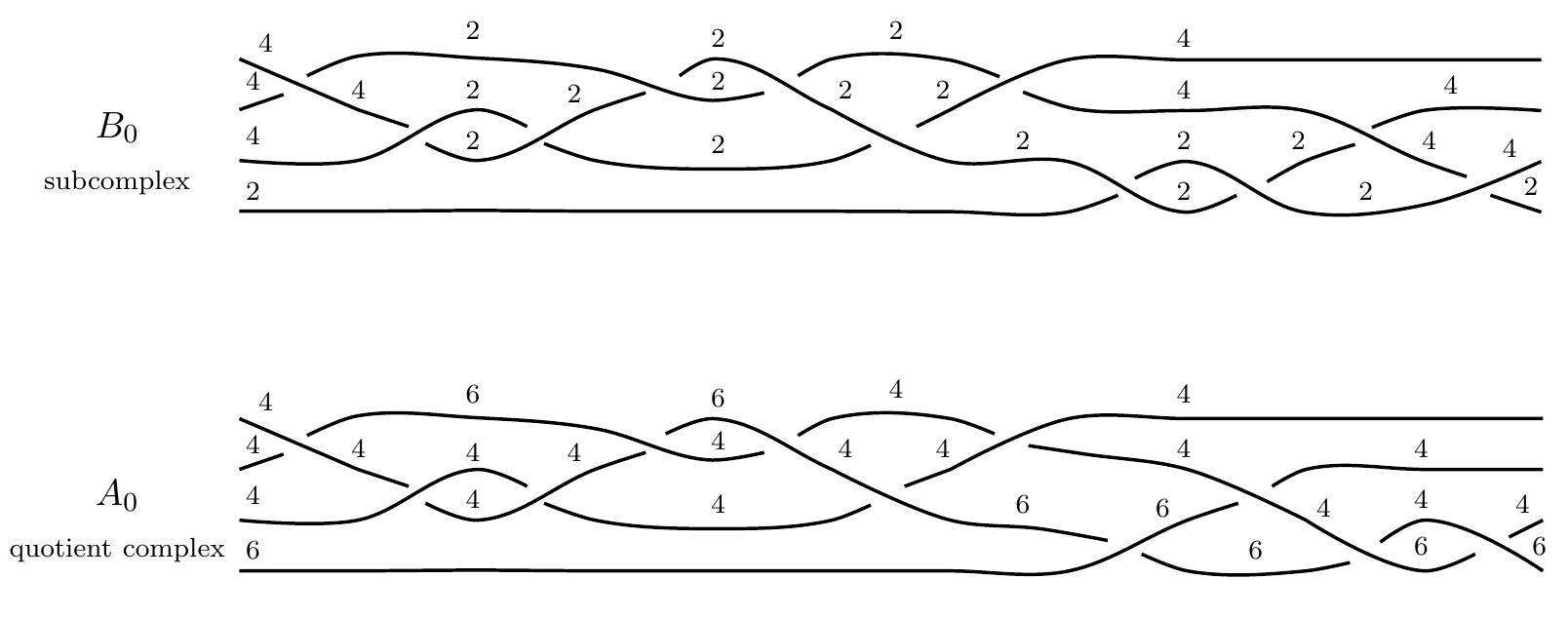}}
\caption{Values of $\quotkapp$ and $\subkapp$ may depend on $p$.  The number above each arc represents a value of $\quotkapp$ or $\subkapp$ based on placing $p$ on that arc.}
\label{fig:reducedinteresting}

\end{figure}
\begin{Ex*}
Recall that Ng and Kandhawit define two infinite families of braids $A(a,b)$ and $B(a,b)$ so that, for any $a, b \in \bz_{\geq 0}$, the closures of $A(a,b)$ and $B(a,b)$ have the same topological knot type and self-linking number but are not transversely isotopic.  Write $A_0$ and $B_0$ for $A(0,0)$ and $B(0,0)$.  We have already seen that $\kappa(A_0) = 4$ and $\kappa(B_0) = 2$.  For any $p \in \bar A_0$ we have $\subkapp(A_0) = 4$ and $\quotkapp(B_0) = 4$.  On the other hand, $\quotkapp(A_0)$ and $\subkapp(B_0)$ depend on $p$.  See Figure \ref{fig:reducedinteresting}.  
\end{Ex*}

\begin{Rem*}
It is straightforward to check that the two candidates for ``reduced annular Khovanov homology'' are not isomorphic (for example with the braid $\sigma_{1} \in B_{2}$).  This fact is not mentioned elsewhere in the literature. The calculations above show that the difference between the two is significant, and that the two reductions might provide different information in light of the $k$-grading.
\end{Rem*}

\footnotesize
\bibliographystyle{amsalpha}
\bibliography{Paper}

\end{document}